\documentclass[12pt]{amsart}
\usepackage{amsmath}
\usepackage{amsxtra}
\usepackage{amscd}
\usepackage{amsthm}
\usepackage{amsfonts}
\usepackage{amssymb}
\usepackage{mathrsfs}
\usepackage[all]{xy}
\usepackage[pdftex]{color,graphicx}
\usepackage{graphicx}
\usepackage{color}
\usepackage{mathrsfs}

\newtheorem{thm}{Theorem}
\newtheorem{lemma}[thm]{Lemma}
\newtheorem{prop}[thm]{Proposition}

\newtheorem{corol}[thm]{Corollary}

\newtheorem*{thmA}{Theorem}

\def\Z{\Bbb Z}

\def\Q{\Bbb Q}

\def\R{\Bbb R}

\def\Hy{\mathcal{H}}

\DeclareMathOperator\aut{Aut}

\DeclareMathOperator\Mod{mod}

\DeclareMathOperator\dist{dist}
\DeclareMathOperator\Orb{\mathcal{O}}

\makeatletter
\newcommand{\rmnum}[1]{\romannumeral #1}
\newcommand{\Rmnum}[1]{\expandafter\@slowromancap\romannumeral #1@}
\makeatother

\begin{document}

\begin{abstract}  Let $\mathrm{O}(f,\Z)$ be the integral orthogonal group of an integral quadratic form $f$ of signature $(n,1)$.
Let $\mathrm{R}(f,\Z)$ be the subgroup of $\mathrm{O}(f,\Z)$ generated by all hyperbolic reflections. 
Vinberg 
\cite{Vi1,Vi3} proved that if $n \ge 30$ then the reflective quotient
$\mathrm{O}(f,\Z)/\mathrm{R}(f,\Z)$ is infinite.  In this note we generalize Vinberg's theorem 
and prove that if $n \ge 92$ then  $\mathrm{O}(f,\Z)/\mathrm{R}(f,\Z)$ contains a non-abelian free group (and thus it is not amenable).

\end{abstract}

\title{On the non-amenability of the reflective quotient \Rmnum{1}: The rational case}


\author{Chen Meiri}
\address{Department of Mathematics,
University of Chicago,
5734 S. University Avenue,
Chicago, Illinois, USA, 60637}
\email{chenmeiri@math.uchicago.edu}
\maketitle
\date{\today}

\section{Introduction} 

Let $f$ be an integral quadratic form of signature $(n,1)$. The integral orthogonal group $\mathrm{O}(f,\Z)$
has an index two subgroup $\mathrm{O}^+(f,\Z)$ which is a discrete group of motions of the $n$-dimensional hyperbolic space. The reflection subgroup $\mathrm{R}(f,\Z)$ is the subgroup
generated by all hyperbolic reflections in $\mathrm{O}^+(f,\Z)$. Since a conjugate of 
a hyperbolic reflection is also a hyperbolic reflection, $\mathrm{R}(f,\Z)$ is a normal subgroup of
of $\mathrm{O}(f,\Z)$. The reflection subgroup $\mathrm{R}(f,\Z)$ and reflective quotient $\mathrm{O}(f,\Z)/\mathrm{R}(f,\Z)$ play an important role in many geometric situations (see \cite{Do} and the reference therein). 
However, the exact structure of $\mathrm{O}(f,\Z)/\mathrm{R}(f,\Z)$ 
is known only for some quadratic forms in relatively small dimensions. 
Let $f_n$ be the unique odd unimodular quadratic form of signature $(n,1)$.
Vinberg \cite{Vi2} found the structure of $\mathrm{O}(f_n,\Z)$ and $\mathrm{O}(f_n,\Z)/\mathrm{R}(f_n,\Z)$  for 
for $n \le 17$. This result was extended by Vinberg and Kaplinskaya to $n \le 19$. In all these cases the reflective quotient is finite. 
Conway \cite{Co} proved that if $q $ is the unique even unimodular quadratic form of signature 
$(25,1)$ then  $\mathrm{O}^+(q,\Z)/\mathrm{R}(q,\Z)$ is isomorphic to the isometry group of the Leech lattice.
Conway's result is an example where the reflective quotient is infinite but amenable.  
Borcherds \cite{Bo2} found the structure of $\mathrm{O}(f_n,\Z)/\mathrm{R}(f_n,\Z)$ for $20 \le n \le 24$, in these case the reflective quotient
is infinite and isomorphic to a direct limit of finite groups. There are several other works which deals with other quadratic form of relatively small dimensions. 

When $n$ is large, the exact structure of the reflective quotient is not known and 
the results are more qualitative. Vinberg's theorem \cite{Vi1,Vi2} implies that  if $n \ge 30$ then the reflective quotient
$\mathrm{O}(f,\Z)/\mathrm{R}(f,\Z)$ is infinite for any $f$. A stronger theorem with respect to odd unimodular quadratic forms $f_n$
was proven by Borcherds \cite{Bo2}; When $n\ge 25$ is congruent to 2,3 or 6 modulo 8 then the reflective quotient $\mathrm{O}(f_n,\Z)/\mathrm{O}(f_n,\Z)$ is a non-trivial amalgamated free product (not necessarily of finite groups). The main theorem of this note is:

\begin{thm}\label{thm main} Let $L \subseteq \Q^{n+1}$ be a lattice and let $f$ be a quadratic form of signature $(n,1)$ which takes integral values on $L$. If $n \ge 92$ then the reflective quotient $\mathrm{O}(f,L)/\mathrm{R}(f,L)$ 
contains a non-abelian free subgroup. 
\end{thm}

By combining Theorem \ref{thm main} with the results of \cite{Br} we get:

\begin{corol} Let $n \ge 92$ and let $f$ be a classically integral quadratic form of signature $(n,1)$.
Let $\mathbb{H}^n$ be the $n$-dimensional hyperbolic space and let
$\Gamma$ be a finite index subgroup of $\mathrm{R}(f,\Z)$. Then the bottom of the $L^2$-spectrum 
of the Riemann manifold $\Gamma \setminus \mathbb{H}^n$ is positive.
\end{corol}

The following proposition is a key ingredient in the proof of Theorem \ref{thm main}.

\begin{prop}\label{prop 41} There exist three non-equivalent even integral positive definite quadratic forms of discriminant $2$ and dimension $41$ which do no represent the number $2$.
\end{prop}

{\flushleft{\bf Remark.}} If $k \le 3$ then there do not exist three non-equivalent 
even integral positive definite quadratic forms of discriminant 
$2$ and dimension $8k+1$ which do not represent 2. Such forms probably exist for $k=4$ and then the bound $92$ in Theorem \ref{thm main}
could be improved to $76$.
\\ \\
{\bf Remark.} In a paper in preparation we extend the results of this note to
general arithmetic lattices  in $\mathrm{O}(n,1)$ for large $n$, which  come from quadratic forms
over number fields, including the case of anisotropic forms.
\\ \\
{\bf Acknowledgment.} The author is thankful to Peter Sarank for asking him if the reflective quotient is always
non-amenable in large enough dimensions and for many helpful discussions.

\section{Proof of Theorem \ref{thm main}}\label{sec proof}
{{\flushleft{\bf Preliminaries.}} 
Let $L \subseteq \Q^n$ be a lattice and let $f$ be a quadratic form of signature $(n,1)$ which is defined on $\R^{n+1}$
and takes integral values on $L$. For every $x,y\in \R^{n+1}$ define 
$(x,y):=\frac{1}{4}(f(x+y)-f(x-y))$. The form $(\cdot,\cdot)$ is bilinear on $\R^{n+1}$.  By replacing $f$ with $2f$ is necessary, we can assume that $f$ is classically integral on $L$. (Classically integral  means that $(v,u)$ is an integer for every $v,u\in L$.)

Let $k$ be a positive integer. Every one of the connected components of the hyperboloid $\{x \in \R^{n+1} \mid f(x)=-k\}$ is a model for the $n$-dimensional hyperbolic space. Let $\Hy$ be one of these components. Let $\mathrm{O}^+(f,L)$ the subgroup  of $\mathrm{O}(f,L)$ preserving $\Hy$. 
The group $\mathrm{O}^+(f,L)$ has index two in $\mathrm{O}(f,L)$ and 
every one of its elements acts as a hyperbolic isometry of $\Hy$. 
An element $v \in L$ is called primitive if $\frac{1}{m}v$ does not belong to $L$ for every $m \ge 2$.
An element $v \in L$ is called isotropic if $f(v)=0$.
A non-isotropic primitive element $v \in L$ is called a root if the linear map $r_v:\R^{n+1} \rightarrow \R^{n+1}$ defined by
\begin{equation}\label{eq reflection}
r_v(x)=x-\frac{2(v,x)}{f(v)}v
\end{equation}
preserves $L$. Since every root $v$ is a primitive element, the requirement that $r_v$ preserves $L$ is equivalent to the requirement that for every $u\in L$, the number
$\frac{2(v,u)}{f(v)}$ is an integer.  
A root $v$ is called positive or negative depending on whether $f(v)>0$ or $f(v)<0$. Let $\mathcal{R}^+(f,L)$ and $\mathcal{R}^-(f,L)$ be the set of positive roots and the set of negative roots respectively. If $v$ is a positive root of $L$ then $r_v$ belongs to $ \mathrm{O}^+(f,L)$ and it is a hyperbolic reflection
in the hyperbolic hyperplane $H_v:=\{x \in \Hy \mid (v,x)=0\}$.
However, if $v$ is a negative root then $r_v$ belongs to $\mathrm{O}(f,L)$ but not to 
$\mathrm{O}^+(f,L)$. In this case, the map $c_v:=-r_v$ belongs to $\mathrm{O}^+(f,L)$ and it is a Cartan involution whose base point $u$ is the intersection of the linear subspace $\{\alpha v \mid \alpha \in \R\} \subseteq \R^{n+1}$ with $\mathcal{H}$. Thus, $c_v$ fixes $u$ and reverses the geodesic passing through $u$.

Let $D^\circ$ be a connected component of $\Hy \setminus \cup_{v \in \mathcal{R}^+(f,L)}H_v$. The closure $D$ of $D^\circ$ is a fundamental domain for $\mathrm{R}(f,L)$ and $D^\circ$ is the interior of $D$. Every $g \in \mathrm{O}^+(f,L)$ permutes the positive roots  and thus also the connected components of 
$\Hy \setminus \cup_{v \in \mathcal{R}^+(f,L)}H_v$. It follows that $\mathrm{O}^+(f,L)$ is the semidirect product 
\begin{equation}\label{eq semi direct}
\mathrm{O}^+(f,L) = \mathrm{R}(f,L) \mathbb{o} \mathrm{Stab}(D^\circ)
\end{equation}
where $\mathrm{Stab}(D^\circ)$ is the stabilizer of $D^\circ$ under the action of $\mathrm{O}^+(f,L)$ on the components. The group $\mathrm{Stab}(D)$ is naturally isomorphic to a subgroup of $\aut(D)$ where $\aut(D)$ is the symmetry group of the polyhedron $D$.
\\ \\
{\flushleft{\bf Step A:}}
We start with the special case $L=\Z^{42}$ and $f=q$ where $q$ is 
the unique (up to equivalence) even unimodular quadratic form of signature $(41,1)$.  We choose $\mathcal{H}:=\{x \in \R^{42} \mid q(x)=-2 \}$ as a model for the 41-dimensional 
hyperbolic space. 
Since $q$ is even and unimodular, an element $v \in L$ is a positive root if and only if $q(v)=2$. Similarly, an element $v \in L$ is a
negative root if and only if $q(v)=-2$. Thus, if $v$ is a negative root then either $v$ or $-v$ is the base point of $-r_v$.
For every element $v \in \Z^{42}$ let 
$\mathcal{O}(v)$ be its orbit under $\mathrm{O}^+(q,L)$.
\begin{lemma}\label{lem 3 elements} If $v,u,w\in D^\circ$ are negative roots and $\Orb(v)$, $\Orb(u)$ and $\Orb(w)$ are distinct then the group generated by the Cartan involutions $c_v$, $c_u$ and $c_v$ contains a non-abelian free group. In particular, if such $v$, $u$ and $w$ exist then $\mathrm{O}(q,L)/\mathrm{R}(q,L)$
contains a free group.
\end{lemma}
\begin{proof}

We start by showing that any geodesic in $\Hy$ intersects non-trivially at most 2 orbits of negative roots. 
Let $l$ be a geodesic which contains at least two negative roots. Choose two negative roots 
$v_1$ and $v_2$ on $l$ with minimal distance.  The negative root $v_3:=c_{v_1}(v_2)$ belongs to $l$ and $v_1$ is between $v_2$ and $v_3$. The Cartan involution $c_v$ fixes $v$ and inverses the geodesics in $v$ so $\dist(v_1,v_2)=\dist(v_1,v_3)=\frac{1}{2}\dist(v_2,v_3)$. The map
$g:=c_{v_1}c_{v_2}$ is a hyperbolic translation along $l$ with translation length $\dist(v_2,v_3)$. If $v_4$ is any negative root on $l$ then there exists $m \in \Z$ such that 
$g^m(v_4)$ belongs to the geodesic segment $[v_2,v_3]\subseteq l$. Since 
$\dist(v_1,v_2)$ is minimal, $g^m(v_4)$ equals to  $v_1$ or $v_2$ or $v_3$. Thus, $v_4$ belongs to $\mathcal{O}(v_1)$ or $\mathcal{O}(v_2)$. Hence, it is not possible that $v$, $u$ and $w$ are contained in the same geodesic. 

Let $l_1$ be the geodesic containing $v$ and $u$ and 
let $l_2$ be the geodesic containing $v$ and $w$. As before,
$g_1:=c_vc_u$ and $g_2:=c_vc_w$ are hyperbolic translations along $l_1$ and $l_2$ with translation lengths $2\dist(v,u)$ and $2\dist(v,w)$. Since $l_1$ and $l_2$ do not meet on the boundary of $\Hy$ there exists an $m$ such that $g_1^m$ and $g_2^m$ generate a Schottky group which is free of rank 2.

The Cartan involutions  $c_v$, $c_u$, and $c_w$ belong to  belong to the stabilizer of $D^\circ$ since their base points belong to $D^\circ$. The last sentence of the Lemma follows from Equation (\ref{eq semi direct}).
\end{proof}

Every negative root $v$ induces an even integral positive definite quadratic form of discriminant 2 on
the orthogonal complement $L_v:=\{u \in L \mid (v,u)=0\}$ of $v$. By identifying $L_v$ with $\Z^{41}$,  $v$ induces 
an equivalence class $Q(v)$ of even integral positive definite quadratic forms of discriminant 2 on $\Z^{41}$. Lemma 3.1.2 of \cite{Bo1} implies that the correspondence $v \mapsto Q(v)$ is a bijection between orbits of negative roots and equivalence classes of even $41$-dimensional positive definite quadratic forms of discriminant 2. 
A negative root $v\in \mathcal{H}$ does not belong to $\cup_{u \in \mathcal{R}^+(q,L)}H_u$ if and only if the quadratic forms in $Q(v)$ do not represent the number 2.
The reflection group $\mathrm{R}(f,L)$ acts transitively on the connected components of  $\cup_{u \in \mathcal{R}^+(q,L)}H_u$.  
Thus, if $v\in \mathcal{H}$ is a negative root then
$\Orb(v)\cap D^\circ \ne \emptyset$ if and only if the quadratic forms in $Q(v)$ do not represent the number $2$. Hence, Proposition \ref{prop 41} implies the existence of $v$, $u$ and $w$ which satisfy the assumptions of Lemma \ref{lem 3 elements}. This completes the proof that $\mathrm{O}(q,\Z^{42})/\mathrm{R}(q,\Z^{42})$ contains a non-abelian free group. 
\\ \\
{\bf Step B:} Let $L=\Z^{n+1}$ and let
$$f(x_1,\ldots,x_{n+1})=\alpha{q}(x_1,\ldots,x_{42}) \oplus t(x_{43},\ldots,x_{n+1})$$
where $t$ is some classically integral positive definite quadratic form
and, $q$ is as in Step A and $\alpha$ is some positive integer. 
Identify $\R^{42}$ with the subspace of $\R^{n+1}$ consisting of the elements with zeros on their last $n-41$ entries. 
Let $\mathcal{H}^n$ be one of the connected  components of 
$\{x \in \R^{n+1} \mid f(x)=-2\alpha\}$ and denote $\mathcal{H}^{41}:=\mathcal{H}^{n}\cap\R^{42}$.
Choose $\mathcal{H}^n$ as a model for the hyperbolic $n$-dimensional space and note that
$\mathcal{H}^{41}$ can be identified with the 41-dimensional hyperbolic space which was considered
in Step A. The group $\mathrm{O}(q,\Z)$ can be identified with the subgroup of $\mathrm{O}(f,\Z)$ consisting of the automorphs  which act as the identity on the last $n-41$ coordinates. Under these identifications, the action of $\mathrm{O}^+(q,\Z)$ on $\mathcal{H}^{41}$ is just its usual action on the 41-dimensional hyperbolic space. If $v \in \Z^{42}$ is a positive root of $q$ then it is also a positive root of $f$. Hence, the inclusion of  $\mathrm{O}(q,\Z)$ in $\mathrm{O}(f,\Z)$ induces a homomorphism $$\rho:\mathrm{O}({q},\Z)/\mathrm{R}(q,\Z)\rightarrow \mathrm{O}(f,\Z)/\mathrm{R}(f,\Z).$$
Thus, it will be enough to show that there exists a fundamental domain $D \subseteq \mathcal{H}^n$ for $\mathrm{R}(f,\Z)$ such that  $D \cap \mathcal{H}^{41}$ is a fundamental domain for $\mathrm{R}(q,\Z)$. Indeed, if this is true then $\rho$ is invective and the result follows from Step A.
In order to show that such fundamental domain $D $ exists, it is enough to prove that for every 
$v \in \mathcal{R}^+(f,L)$ either $H_v \cap \mathcal{H}^{41}=\emptyset$ or $H_v \cap \mathcal{H}^{41}=\mathcal{H}^{41}$ or 
$H_v \cap \mathcal{H}^{41}=H_{w}$ for some $w \in \mathcal{R}^+(q,\Z^{42})$.

Fix $v \in \mathcal{R}^+(f,L)$ and let $u$ be its projection into $\R^{42}$. If $u=0$ then $H_v \cap \mathcal{H}^{41}=\mathcal{H}^{41}$. If $u \ne 0$ and
$q(u)\le 0$ then $H_v \cap \mathcal{H}^{41}=\emptyset$. We are left to deal with the case
$q(u)>0$.
Since $v$ is a root, for every $z \in L$
 $$\frac{2(v,z)}{\alpha q(u)+t(v-u)}=\frac{2(v,z)}{f(v)} \in \Z$$  where as before $(\cdot,\cdot)$ is the bilinear form induced by $f$. By  taking $z=u$ we get $$\frac{2\alpha q(u)}{\alpha q(u)+t(v-u)}=\frac{2(v,u)}{f(v)} \in \Z.$$ Since $t$ is positive definite, there are only two possibilities. Either $t(v-u)=\alpha q(u)$ or $t(v-u)=0$ (and then $v=u\in \R^{42}$). It follows that in both cases  $\frac{2(u,z)}{\alpha q(u)}$ is integral for every $z \in \Z^{42}$. Hence, $u$ is an integral multiply of some positive root $w \in \mathcal{R}^+(q,\Z^{42})$ and  $H_v \cap \mathcal{H}^{41}=H_{w}$. 
\\ \\
{\bf Step C:} This step deals with the general case. The arguments closely follows the ones in 
Vinberg's proof \cite{Vi3}. Let $L \subseteq \Q^n$ be a lattice and let $f$ be a quadratic form of signature $(n,1)$ defined on $\R^{n+1}$ such that $f$ is classically integral on $L$. Assume for the moment that there exists a lattice $\tilde{L}$ 
with the following properties:
\begin{itemize}
\item  $L \subseteq \tilde{L} \subseteq \Q^{n+1}$
\item $\tilde{L}$ is preserved by $\mathrm{O}(f,L)$.
\item $f$ is classically integral on $L$. 
\end{itemize}

The group  $\mathrm{O}(f,L)$ is a finite index subgroup of $\mathrm{O}(f,\tilde{L})$. Since 
every root of $L$ is also a root of $\tilde{L}$, the inclusion induces 
a homomorphism $$\rho:\mathrm{O}(f,L)/\mathrm{R}(f,L)\rightarrow \mathrm{O}(f,\tilde{L})/\mathrm{R}(f,\tilde{L})$$ and the image of $\rho$ has finite index
in $ \mathrm{O}(f,\tilde{L})/\mathrm{R}(f,\tilde{L})$. 
Thus, it is enough to prove the result for $\tilde{L}$ and $\tilde{F}$. 
The main idea is to find $\tilde{L}$ for which $f$ is an orthogonal sum as in Step B.
 
Let  
$$L^*:=\{v \in \Q^{n+1} \mid \forall \ u \in L. \ (v,u)\in \Z  \}$$
be the adjoint lattice of $L$. 
The group $L^*/L$ is a finite abelian
group so it is a direct sum of its $p$-Sylow subgroups. Each $p$-Sylow subgroup $P$ is a direct sum of cyclic $p$-groups. The sizes of the cyclic groups are uniquely determined by $P$ and called the invariant $p$-factors of $L$. If $L$ has an invariant $p$-factor $p^m$ for some prime $p$
and $m \ge 2$ then we can replace $L$ with $$\hat{L}:=\{v+u \mid v \in L\ \&\ u\in pL^* \cap p^{-1}L\}.$$ The form $f$ is still classically integral on $\hat{L}$ and $\mathrm{O}(f,L)$ preserves $\hat{L}$. In addition, $\hat{L}$  has a smaller discriminant than $L$ so after a finite 
number of such replacements we will get a lattice $\tilde{L} $ with the following properties: 
\begin{itemize}
\item[(\rmnum{1})] $L \subseteq \tilde{L}  \subseteq \Q^{n+1}$ and $f$ is classically integral on $\tilde{L}$
\item[(\rmnum{2})] $\mathrm{O}(f,L)$ preserves  $\tilde{L}$ and has a finite index in $\mathrm{O}(f,\tilde{L})$
\item[(\rmnum{3})] For every prime $p$, every invariant $p$-factor of $\tilde{L}$ equals to 1 or p
\end{itemize}

We choose a $\Z$-base for $\tilde{L}$ and identify it with $\Z^{n+1}$. We regard $f=f(x_1,\ldots,x_{n+1})$ as given in this base.  
Part (\rmnum{3}) implies that every form in the genus of $f$ is properly equivalent to $f$ (Theorem 1.5 of Chapter 11 in \cite{Ca}). Thus, if $h$ is a classically integral quadratic form
of signature $(m,1)$ and for every prime $p$ there exists a $p$-adic integral form $s_p$ such that $f$ is $\Z_p$-equivalent to $h \oplus s_p$, then there exists an integral positive definite quadratic form $s$ such that $f$ is $\Z$-equivalent to $h \oplus s$. 
Part (\rmnum{3}) implies that for every prime $p$ there are $p$-adic integral forms $s_{1,p}$ and $s_{2,p}$ whose discriminates are units
such that 
$$f(x_1,\ldots,x_{n+1})=s_{1,p}(x_1,\ldots,x_l) \oplus ps_{2,p}(x_{l+1},\ldots,x_{n+1})$$
and the dimension of at least one of them is greater than $46$ (since $n \ge 92$).  If the dimension of $s_{1,p}$ is grater than $46$ define $\delta_p:=0$, otherwise define $\delta_p:=1$. Let $q$ be the unique even unimodular quadratic from of signature $(41,1)$ and define 
$\alpha=\prod_{p}p^{42\delta_p}$.  

Let $p $ be a prime and denote $\alpha_p:=\alpha/p^{42\delta_p}$, $\alpha_p$ is a $p$-adic unit. Assume first that $p \ne 2$.
Every $p$-adic quadratic form of dimension $m$ with a unit discriminant $d$ is $\Z_p$-equivalent to 
$$x_1^2+\cdots+x_{m-1}^2+dx_m^2.$$
Since $47 \ge 42+1$, there exists 
an integral $p$-adic form $w_{\delta_p,p}$ such that  $s_{\delta_p,p}$ is $\Z_p$-equivalent
to  $\alpha_p q \oplus w_p$.

We are left to deal with the case $p=2$. Every integral $2$-adic quadratic form of dimension at least 5 is isotropic. 
Thus, if $h$ is a 2-adic quadratic form with a unit discriminant and dimension at least 5, then $h$ is equivalent to the sum $h_1+h_2$ where $h_1$ has dimension 2 and is equal to either $2x_1x_2$ or to $2x_1x_2+x_2^2$. Moreover, $2x_1x_2+(2x_3x_4+x_4^2)$ is equivalent to 
$(2x_1x_2+x_2^2)+(2x_3x_4+x_4^2)$. Since $q$ is $\Z_2$-equivalent to the sum of $21$ 
copies of $2x_1x_2$ and $47 = 2\cdot 22+3$, a straightforward argument shows that there exists a classically integral $2$-adic form $w_{2}$ such that  $s_{\delta_2,2}$ is $\Z_2$-equivalent
to  $\alpha_2 q \oplus w_2$.

Hence, for every prime $p$, $f$ is $\Z_p$-equivalent $\alpha q \oplus t_p$ for some  $p$-adic classically integral form $t_p$. Thus, 
$f$ is equivalent to $\alpha q \oplus t$ for some classically integral positive definite  quadratic from
$t$ and we can use Step $B$. This proof of Theorem \ref{thm main} is now complete.

\section{Proof of Proposition \ref{prop 41}}\label{sec prop}

There exist a vast literature about unimodular positive definite quadratic forms which do not represent small positive integers. Let us mention two such results. Conway and Thompson (Theorem 9.5 in \cite{MH}) proved that for every positive integer $k$ there exists a positive number $n(k)$ such that for every $n \ge n(k)$ there exists a unimodular positive definite quadratic form of dimension $n$ which does not represent any integer smaller than $k$. Their proof uses a simple form of Siegel's mass formula and the computations
are not complicated. King \cite{Ki} 
proved that  there are at least 10,000,000 non-equivalent even unimodular positive definite quadratic forms of dimension 32 which do not represent the number 2. King's proof is based on the general form of Siegel's mass formula and the computation are more involved and done by a computer. 

We will prove Proposition \ref{prop 41} by combining the argument of Conway and Thomson with King's theorem. This strategy allows us to avoid tedious computations. While the   
the statement if Proposition \ref{prop 41} is far from optimal, it is suitable for our needs.

Some preparation is needed before stating Siegel's mass formula.
Let $f$ be a positive definite integral quadratic form of dimension $n$. For every prime $p$ and every integer $m$, the $p$-density of $f$ at $m$ is defined to be
\begin{equation}\label{p denstiy}
Df_p^{-1}(m):=\lim_{k \rightarrow \infty} 
\frac{|\{\bar{x} \in (\Z/p^k\Z)^n \mid f(\bar{x}) = m\ \Mod\ (p^k)\}|}{p^{k(n-1)}}.
\end{equation}

For every $m$ there exists $k_0$ such that if $k \ge k_0$ then the fraction inside the limit does not depend on $k$ so the limit is a rational number. The $\infty$-density at a point $y>0$ is defined to be  
\begin{equation}\label{inf denstiy}
Df_\infty^{-1}(y):=\lim_{r \rightarrow 0} 
\frac{\mu_{n}(f^{-1}(B_r(y)))}{\mu(B_r(y))}
\end{equation}
where $B_r(m)$ is the ball of radius $r$ around $y$ and $\mu$ and $\mu_{n}$ the usual 
Lebesgue measures on $\R$ and $\R^{n}$. 
(It is also possible to define the $p$-density in a similar manner to Equation (\ref{inf denstiy}) by replacing the Lebesgue measure with the normalized Haar measures on $\Z_p$ and $\Z_p^{n}$. 
It is easily verified that this definition is equivalent to the one given by Equation (\ref{p denstiy}).) The $\infty$-density is continuous as a function of $y$ and 
\begin{equation}
\int_0^y \mathrm{D}f_\infty^{-1}(t)\mathrm{d}t=\mu_n(f^{-1}[0,y])=\mu_n(B_{\sqrt{y}}(0))=\frac{1}{\sqrt{\mathrm{disc(f)}}}\omega_ny^{\frac{n}{2}}
\end{equation}
where $\mathrm{disc}(f)$ is the discriminant of $f$ and 
\begin{equation}\label{omega}
\omega_n:= \frac{\pi^{\frac{n}{2}}}{\Gamma(\frac{n}{2}+1)}.
\end{equation}

Thus, for $y>0$,
\begin{equation}
\mathrm{D}f_\infty^{-1}(y)=\frac{n}{2\sqrt{\mathrm{disc(f)}}}\omega_ny^{\frac{n}{2}-1}
\end{equation}

Note that the $p$-density $Df_p^{-1}(m)$ and the $\infty$-density $Df_\infty^{-1}(m)$ depend only on the genus of $f$ but not on $f$ itself.

The are finitely many equivalence classes $\mathcal{G}_1,\ldots, \mathcal{G}_r$ of quadratic forms in the genus of $f$. 
If two quadratic forms belong to the same equivalence class $\mathcal{G}_i$ then their integral orthogonal groups have the same size which we denote by 
$|\mathrm{O}(\mathcal{G}_i)|$. The weight of $\mathcal{G}_i$ is defined to be:

\begin{equation}
w_i:=|\mathrm{O}(\mathcal{G}_i)|^{-1}\left/ \sum_{j=1}^r|\mathrm{O}(\mathcal{G}_j)|^{-1}\right..
\end{equation} 

Let $m$ be a positive integer and let $g \in \mathcal{G}_i$. The number of times that $g$ represents $m$ (i.e., the size of $\{v \in \Z^n \mid g(v)=m \}$) depends  only on $\mathcal{G}_i$ and is denoted by $r_i(m)$. 
(Since $f$ and thus also $g$ are positive definite $r_i(m)$ is finite.)

Siegel's mass formula expresses the weighted-average number of integral representation of an integer $m$ by quadratic forms in the genus of $f$
in terms of the primes and infinite densities.

\begin{thmA}[Siegel's mass formula, \cite{Si}]\label{Siegel}  For every $m \ne 0$, 
\begin{equation}
\sum_{i=1}^r w_ir_i(m)=\varepsilon \mathrm{D}f_\infty^{-1}(m)\prod_{p}\mathrm{D}f_p^{-1}(m)
\end{equation} 
where the product is over all the prime numbers and  $\varepsilon=\frac{1}{2}$ if $n=2$ and $\varepsilon=1$ if $n \ge 3$.

\end{thmA}

{\flushleft{\emph{Proof of Proposition \ref{prop 41}. }}} The first part of the proof is almost identical to the proof of the theorem of Conway and Thompson 
which was stated above (Theorem 9.5 in \cite{MH}).

Let $f$ be some fixed integral even positive definite quadratic forms of discriminant 2 and dimension 41 (which might represents 2). When computing the $p$-density of $f$ we can regard it as a $\Z_p$-integral quadratic form. 
The $p$-density  $\mathrm{D}_p^{-1}(m)$ does not depend on $f_p$ itself but only on the $\Z_p$-equivalence class of $f_p$. If $f$ is equivalent to the sum $g+h$ for some integral $p$-adic quadratic forms $g$ and $h$, then the proof of Lemma 9.4 in Chapter 2 of \cite{MH} implies that 
\begin{equation}
\sup_{m \in \Z}\mathrm{D}f_p^{-1}(m) \le \sup_{m \in \Z}\mathrm{D}g_p^{-1}(m). 
\end{equation}

Let $p \ne 2$ be a prime. It follows form the structure theorem of integral $p$-adic 
quadratic forms (Theorem 3.1 of Chapter 8 in \cite{Ca}) that $f$ (when regraded as an integral $p$-adic form) is equivalent to the sum of 
\begin{equation}
g(x_1,\ldots,x_8):=x_1^2+\cdots+x_8^2
\end{equation}
and some other integral $p$-adic quadratic form. The proof of Lemma 9.4 in Chapter 2 of \cite{MH} implies that for every prime $p\ne 2$
\begin{equation}
\mathrm{D}f_p^{-1}(2)  \le\sup_{m \in \Z}\mathrm{D}g_p^{-1}(m) \le \frac{1-p^{-4}}{1-p^{-3}}
\end{equation}
so 
\begin{equation}\label{bound primes}
\sup_{m \in \Z}\prod_{p\ne 2}\mathrm{D}f_p^{-1}(m)  \le \frac{14\zeta(3)}{15\zeta(4)}\le \frac{11}{10}.
\end{equation}

Assume that $p = 2$. It follows form the structure theorem of integral $2$-adic 
quadratic forms (Lemma 4.1 of Chapter 8 in \cite{Ca}) that $f$ (when regraded as an integral $2$-adic form) is equivalent to the sum of 
\begin{equation}
h(x_1,x_2):=2x_1x_2
\end{equation}
and some other integral $2$-adic quadratic form. It is easy to see that if $m \in \Z$ is even then $\mathrm{D}h^{-1}(m)=2$ while if $m \in \Z$ is odd then $\mathrm{D}h^{-1}(m)=0$. Thus, 
\begin{equation}\label{bound 2}
\sup_{m \in \Z}\mathrm{D}f_2^{-1}(m)  \le 2.
\end{equation}

It is left to bound  $\mathrm{D}f_\infty^{-1}(2)$.
By Stirling's approximation
\begin{equation}
\omega_n\le \frac{1}{\sqrt{n\pi}}\left(\frac{2\pi e}{n}\right)^{\frac{n}{2}}
\end{equation}
 Equations (\ref{inf denstiy}) and (\ref{omega}) imply that for $n=41$
\begin{equation}\label{bound inf}
\mathrm{D}f_\infty^{-1}(2)=\frac{n}{2\sqrt{\mathrm{disc(f)}}}\omega_ny^{\frac{n}{2}-1} \le \sqrt{2^{n-5}\cdot\frac{n}{\pi}\cdot\left(\frac{2\pi e}{n}\right)^{n}}\le \frac{1}{50}.
\end{equation}
Equations (\ref{bound primes}), (\ref{bound 2}) and (\ref{bound inf}) imply
\begin{equation}\label{CT}
\mathrm{D}f_\infty^{-1}(2)\prod_{p}\mathrm{D}f_p^{-1}(2) \le \frac{1}{20}.
\end{equation}

Let $g_1,\ldots,g_k$ be representatives for the equivalence classes of even unimodular positive definite quadratic forms of dimension 32 which do no represents 2. 
King \cite{Ki} proved  that 
\begin{equation}
\sum_{i=1}^k |\mathrm{O}(g_i,\Z)|^{-1}=\frac{10968923}{2}.
\end{equation}
   
Let $e_8$ be the unique (up to equivalence) even unimodular positive definite quadratic form of dimension
8 ($e_8$ corresponds to the root system $\mathrm{E}_8$). For every $1 \le i \le k$ define 
\begin{equation}
f_i(x_1,\ldots,x_{41})=g_i(x_1,\ldots,x_{32})+e_8(x_{33},\ldots,x_{40})+2x_{41}^2.
\end{equation}

Each $f_i$ is an even integral positive definite quadratic form of discriminant 2 and dimension 41. Since there exists a unique genus of even umimodular positive definite quadratic forms of dimension 32 all the $f_i$'s belong to the same genus. Recall that an integral quadratic form is called indecomposable if it not equivalent to the sum of two integral quadratic forms of positive dimensions. For example, the $g_i$'s are indecomposable since every even unimodular quadratic forms of dimension smaller than $24$ represents the number 2.
It is clear that every quadratic forms is equivalent to the sum of indecomposable 
quadratic forms.
A theorem of Eichler (Theorem 6.4 of Chapter 2 in \cite{MH}) states the if $f$ is positive definite, then the number of times each equivalence class of indecomposable quadratic forms appears in the  decomposition of $f$ is uniquely determined. Thus,
if $1 \le i \ne j \le k$, then $f_i$ and $f_j$ are not equivalent. Moreover, $e_8$ is 
indecomposable so the integral  orthogonal group of $f_i$ is isomorphic to the direct product of the integral orthogonal groups of $g_i$, $e_8$ and $2x$. 
Since $|\mathrm{O}(e_8,\Z)|=696729600$
  
\begin{equation}
M_1:=\sum_{i=1}^k |\mathrm{O}(f_i,\Z)|^{-1}=\frac{10968923}{4|\mathrm{O}(e_8,\Z)|} \ge \frac{3}{1000}.
\end{equation}

Complete $f_1,\ldots,f_k$ to a representative set $f_1,\ldots,f_{k+t+s}$ of the equivalence classes in the genus of $f_1$ and assume that all the forms $f_{k+1},\ldots,f_{k+t}$ represent the number 2 and all the forms $f_{k+t+1},\ldots,f_{k+t+s}$ do not represent the number 2. Our goal is to show that $s \ge 3$.

Define 
\begin{equation}
M_2:=\sum_{i=k+1}^{k+t} |\mathrm{O}(f_i,\Z)|^{-1} \ \ \text{and} \ \ M_3:=\sum_{i=k+t+1}^{k+t+s} |\mathrm{O}(f_i,\Z)|^{-1}
\end{equation}

The number of times $e_8$ represents the number 2 is 240.  
Siegel's mass formula together with Equation (\ref{CT}) imply that
\begin{equation}
\frac{242M_1+2M_2}{M_1+M_2+M_3} \le \frac{1}{20}.
\end{equation}
Thus,
\begin{equation}
\frac{7}{10}\le 241M_1 \le \frac{1}{20}M_3 \le \frac{s}{40}
\end{equation}
and  $s \ge 28$. The proof of Proposition 
\ref{prop 41} is now complete.


\begin{thebibliography}{1}
\bibitem[Bo1]{Bo1} R.~Borcherds, {\it The Leech lattice and other lattices}. Ph.D. Thesis, Cambridge, (1985). 114 pp. 
\bibitem[Bo2]{Bo2}  R.~Borcherds, {\it Automorphism groups of Lorentzian lattices.} J. Algebra {\bf 111} (1987), no. 1, 133--153.
\bibitem[Br]{Br} R.~Brooks, {\it
The bottom of the spectrum of a Riemannian covering.}
J. Reine Angew. Math. {\bf 357} (1985), 101--114. 
\bibitem[Ca]{Ca} J.W.S.~ Cassels, {\it Rational quadratic forms.} London Mathematical Society Monographs, {\bf 13}. Academic Press, Inc. London-New York, (1978). \rmnum{16}+413 pp.
\bibitem[Co]{Co} J.H.~Conway, {\it The automorphism group of the 
26-dimensional even unimodular Lorentzian lattice.} J. Algebra {\bf 80} (1983), no. 1, 159--163. 
\bibitem[Do]{Do} I.~Dolgachev, {\it Reflection groups in algebraic geometry.} 
Bull. Amer. Math. Soc. (N.S.) {\bf 45 }(2008), no. 1, 1--60. 
\bibitem[Ki]{Ki} O.D.~King, {\it A mass formula for unimodular lattices with no roots.} Math. Comp. {\bf 72} , no. 242, 839--863 (2003) (electronic).
\bibitem[MH]{MH} J.~Milnor  and D.~Husemoller, {\it Symmetric bilinear forms}. Ergebnisse der Mathematik und ihrer Grenzgebiete, Band {\bf 73}. Springer-Verlag, New York-Heidelberg, (1973). \rmnum{8} +147 pp.
\bibitem[Si]{Si} C.L.~Siegel, {\it Gesammelte Abhandlungen \Rmnum{1}.} 
(German) Herausgegeben von K. Chandrasekharan und H. Maass., Springer-Verlag, Berlin-New York (1966)  \rmnum{5}+548 pp.
\bibitem[Vi1]{Vi1} E.B.~Vinberg, {\it On groups of unit elements of certain quadratic forms}. Math. USSR Sbornik {\bf 16}, No.1, pp. 17--35 (1972)
\bibitem[Vi2]{Vi2} E.B.~Vinberg, {\it The groups of units of certain quadratic forms.} (Russian) Mat. Sb. (N.S.) {\bf 87 (129)} (1972), 18--36.
\bibitem[Vi3]{Vi3} E.B.~Vinberg, {\it The absence of crystallographic reflection groups in Lobachevsky spaces of large dimension}. Trudy Moscow. Mat. Obs. {\bf 47}  (1984), 68--102.
\end{thebibliography}
\end{document}